\documentclass[11pt]{article}
\usepackage{amsmath,amssymb,amsthm,verbatim}
\usepackage{color,graphics,srcltx}
\usepackage{hyperref}
\newtheorem{theorem}{Theorem}[section]

\newtheorem{proposition}[theorem]{Proposition}
\newtheorem{corollary}[theorem]{Corollary}
\newtheorem{lemma}[theorem]{Lemma}

\newtheorem{remark}[theorem]{Remark}

\numberwithin{equation}{section}

\def\IF{{\mathbb F}}

\def\IC{{\mathbb C}}
\def\IR{{\mathbb R}}
\def\IA{{\mathbb A}}

\def\cN{{\mathcal N}}

\def\IH{{\mathbb H}}
\def\bV{{\mathbf V}}
\def\tr{\mathop{\rm tr}}
\def\span{{\mathop{\mathrm{Span}}\nolimits}}

\def\bV{{\mathbf V}}

\def\diag{{\rm diag}}

\def\Span{\mathop{\mathrm{Span}}\nolimits}
\begin{document}
\openup 1 \jot

\title{(Real)linear preservers of multiples of unitaries and \\
matrix pairs with some extremal norm properties}

\author{Bojan Kuzma, Chi-Kwong Li, Edward Poon}
\date{}
\maketitle
\begin{abstract}  We determine the structure of linear maps on complex (real) square matrices
sending unitary (orthogonal) matrices to multiples of unitary (orthogonal) matrices. The result
is used to determine the linear preservers of
matrix pairs satisfying the extremal norm properties $\|AB\| = \|A\| \|B\|$,
$\|A^*B\| = \|A\| \|B\|$, or $\|AB^*\| = \|A\| \|B\|$,
for the spectral norm $\|\cdot\|$.

\end{abstract}



\section{Introduction}

Let $M_n(\IF)$ be the set of $n\times n$ matrices over $\IF$, where $\IF$ is the complex field  $\IC$ or the real field $\IR$.
    Denote by ${\mathcal U}_n(\IF)$ the set of $A \in M_n(\IF)$ such that $A^*A = I_n$. Here
$A^*$ is the conjugate transpose of $A$ if $\IF = \IC$ and $A^*$ is the transpose $A^T$ of $A$ if $\IF = \IR$.
We will show that, when $n \ge 3$, a bijective linear map sending matrices in ${\mathcal U}_n(\IF)$ to multiples of matrices in
${\mathcal U}_n(\IF)$ has the form
$$A \mapsto r UAV \quad \hbox{ or } \quad A \mapsto r UA^TV$$
for some nonzero $r \in \IF$ and $U, V \in {\mathcal U}_n(\IF)$.

Denote by $\|\cdot\|$ the spectral norm on $M_n(\IF)$ defined by $\|A\|  = \max\{ \|Ax\|: x \in \IC^n, \|x\| = 1\}$ (recall that $\|A\|$ is the maximal singular value of $A$).
Using the above result, we determine the structure of  \emph{bijective} linear maps $T\colon M_n(\IF)\rightarrow M_n(\IF)$
preserving matrix pairs
satisfying some extremal norm properties. For example, when $n \ge 3$
and  $T\colon  M_n(\IF)\rightarrow M_n(\IF)$ is a bijective linear map, we show that

\begin{itemize}
\item[(a)]  $\|T(A) T(B)\| = \|A\| \|B\|$ whenever  $A, B \in M_n(\IF)$ satisfy
$\|AB\| = \|A\| \|B\|$ if and only if $T$ has the form
$$A \mapsto r UAU^* $$
for some nonzero $r \in \IF$ and $U \in {\mathcal U}_n(\IF)$;
\item[(b)] $T$ has the form
$$A \mapsto r UAV  $$
for some nonzero $r \in \IF$ and $U, V \in {\mathcal U}_n(\IF)$ if and only if any of the following holds.

\smallskip
(b.1) $\|T(A)^*T(B)\| = \|T(A)\| \|T(B)\|$ whenever $\|A^*B\| = \|A\| \|B\|$,

\smallskip
(b.2)   $\|T(A)T(B)^*\| = \|T(A)\| \|T(B)\|$ whenever $\|AB^*\| = \|A\| \|B\|$.
\end{itemize}

\medskip
The case when $n = 2$ is also treated in a separate section since it requires different arguments and  the structure of the maps is more complicated.

People often consider real-linear maps on finite-dimensional complex normed spaces
because distance preserving maps
$f$, i.e., maps satisfying $\|x-y\| = \| f(x)-f(y)\|$ are always
invertible real-linear up to a translation (see, e.g.,~\cite[p.~500] {BhatiaSemrl1997} and \cite{Nica2012} for a very short proof).
In this paper, we will also  determine real-linear maps on complex matrices sending multiples of unitary matrices to
multiples of unitary matrices, or preserving extremal norm properties mentioned above. As we shall see
in Section 3, some proofs can be simplified
if we impose the complex linearity assumption.

Note that related studies on maps that leave the group ${\mathcal U}_n(\IF)$ invariant (rather than mapping it into
$\IF\,{\mathcal U}_n(\IF)$) were carried out in
\cite{Marcus1959, Wei1975}, and a unified approach was given in \cite{BottaPierce1977}.
In particular, \cite{Marcus1959, Wei1975} did not assume bijectivity. In our study, 
the bijectivity is imposed to rule out degenerate maps.   For example,
for any linear functional $f$ on $M_n(\IF)$ and any  $Z \in {\mathcal U}_n(\IF)$,
the singular map $T$ defined by $T(A) = f(A) Z$ will send any matrix  $A \in {\mathcal U}_n(\IF)$ to a multiple of
$Z \in {\mathcal U}_n(\IF)$,
and satisfy
$\|T(A)T(B)\| = \|T(A)\| \|T(B)||$, \ $\|T(A)^* T(B)\| = \|T(A)^*\| \|T(B)\|$,  and $\|T(A) T(B)^*\| = \|T(A)\| \|T(B)\|$
for all $A, B \in M_n(\IF)$. It would be nice to give complete descriptions of the preservers in our study without the 
bijective assumption.

In our discussion, we shall write $e_j$ for the column vector (i.e., an $n$--by--$1$ matrix)
whose only nonzero entry is a one in the $j$th entry, and set $E_{ij} = e_i e_j^*$.

\section{Real-linear preservers on  matrices of size at least 3}

\subsection{Real-linear preservers of multiples of unitary/orthogonal matrices}

The main theorem of this section is the following.

\begin{theorem}\label{thm:main1} Let $n\ge 3$ and let $T\colon M_n(\IF)\to M_n(\IF)$ be a real-linear map.
The following conditions are equivalent.
\begin{itemize}
\item[{\rm (i)}] $T$ is bijective and maps matrices in ${\mathcal U}_n(\IF)$ to multiples of
matrices in ${\mathcal U}_n(\IF)$.
\item[{\rm (ii)}]
There are $U,V\in {\mathcal U}_n(\IF)$ and $r>0$ such that $T$ takes one of the following four forms
      $$X \mapsto rUXV  \quad \hbox{ or } \quad X \mapsto rUX^TV  \quad \hbox{ or } \quad X \mapsto rU\overline{X}V \quad \hbox{ or } \quad X \mapsto rU\overline{X}^TV.$$
    \end{itemize}
      Moreover, if  $T\colon M_n(\IC)\to M_n(\IC)$ is complex-linear then the last two  forms do not exist.
\end{theorem}

The bulk of the proof is a reduction to rank-one preservers. We do it in a separate lemma. By $\IH_n$ we denote the real-linear space of hermitian matrices.

\begin{lemma}\label{lem:rk-1-reduction}
Let $n\ge 3$. A bijective real-linear map $T\colon M_n(\IF)\to M_n(\IF)$ which  maps ${\mathcal U}_n(\IF)$ to $\IF\,
{\mathcal U}_n(\IF)$ preserves the set of rank one matrices.
\end{lemma}
\begin{proof}
  Suppose $A$ is rank one with norm one such  that $T(A) = B$ has rank~$k$. We may replace
$T$ by $V_1 T(U_1 X U_2) V_2$ for suitable $U_1, U_2, V_1, V_2 \in {\mathcal U}_n(\IF)$
and assume that $T(E_{11}) =D = D_1 \oplus 0_{n-k}$, where
$D_1 = \diag(d_1, \dots, d_k)$
with $d_1\ge\dots\ge d_k > 0$.  Let $X = [0] \oplus X_1$ with $X_1 \in U_{n-1}(\IC)$, and $Y = T(X)
= \begin{pmatrix} Y_{11} & Y_{12} \cr Y_{21} & Y_{22} \cr\end{pmatrix}$ with $Y_{11} \in M_k(\IC)$.
Then $T(E_{11}) +\mu T(X) = D +\mu Y$ are multiples of unitary (orthogonal if $\IF=\IR$) matrices for every  $\mu\in\{-1,1\}$.
In particular, $(D +  Y^\ast)(D +  Y) = \alpha_1 I$
and $(D -  Y^\ast)(D -  Y) = \alpha_2 I$. Notice that $(D \pm Y^\ast)(D \pm  Y)$ are positive semidefinite, so $\alpha_1,\alpha_2\ge 0$.
By subtracting the two equations we get  $DY+Y^\ast D =\beta I$ where $\beta= (\alpha_1-\alpha_2)/2 I\in\IR$. Due to $D=D_1\oplus 0_{n-k}$ we easily compute
\begin{equation}\label{eq:DY}
DY+Y^{\ast}D =
\begin{pmatrix} D_1 Y_{11} + Y_{11}^* D_1 & D_1 Y_{12} \cr
Y_{12}^* D_1 & 0_{n-k}\cr \end{pmatrix} = \beta I.
\end{equation}
Note that we also have
$(D+Y)(D+Y^*) = \alpha_1 I$ and $(D-Y)(D-Y^*) = \alpha_2 I$, which gives
\begin{equation}\label{eq:DY*}
  DY^*+YD = \begin{pmatrix} D_1 Y_{11}^* + Y_{11} D_1 & D_1 Y_{21}^* \cr
Y_{21} D_1 & 0_{n-k}\cr \end{pmatrix} = \beta I.
\end{equation}
If $k = n$, then, by~\eqref{eq:DY}, $Y_{11} = Y$ satisfies  $DY+Y^\ast D =\beta I$, or equivalently,  $(DY-\beta/2 I)+(DY-\beta/2 I)^\ast=0$. Since, in this case, $D$ is invertible, we see that the set of all such $Y$  equals  $D^{-1}\bigl(\IA_n(\IF)+\IR I\bigr)$, where  $\IA_n(\IF)$ is the space of skew-hermitian (skew-symmetric if $\IF=\IR$) matrices.
If $\IF=\IC$ we can find  $2(n-1)^2$ (respectively, $(n-1)^2$ if $\IF=\IR$) real-linearly independent matrices $X_1 \in U_{n-1}(\IF)$
-
-
 so that $T([0]\oplus X_1)= Y$
satisfies $DY+Y^*D \in \IR I$.
So, $T$ maps the real-linear space
$${\mathbf V} = \span_\IR \{ [0] \oplus X: X \in U_{n-1}(\IF)\}=0\oplus M_{n-1}(\IF)$$ into a space of dimension bounded by $n^2 + 1$,
i.e., $2(n-1)^2 \le n^2+1$ (respectively, into a space of dimension bounded by $n(n-1)/2 + 1$,
i.e., $(n-1)^2 \le n(n-1)/2+1$, if $\IF=\IR$), which is impossible unless $n \le 3$.

If $k < n$, then $\beta = 0$, and~\eqref{eq:DY}--\eqref{eq:DY*} imply  $D_1Y_{11} + Y_{11}^*D_1 = D_1 Y_{11}^* + Y_{11} D_1 = 0_k$
and $Y_{12} = Y_{21}^* = 0_{k,n-k}$.
Then $T$ will map $\bV$ into a real-linear subspace $\tilde \bV$ consisting of matrices of the form
$Y_{11} \oplus Y_{22}$ with $D_1Y_{11} + Y_{11}^* D_1 = 0_k$.
Now, $\dim_{\IR}\tilde \bV=k^2 + 2(n-k)^2$  (respectively, $\dim_{\IR}\tilde \bV=k(k-1)/2 + (n-k)^2$ if $\IF=\IR$).
So, if $\IF=\IC$ we have $2(n-1)^2 \le k^2+2(n-k)^2 $
(respectively, if $\IF=\IR$ we have $(n-1)^2 \le  k(k-1)/2+(n-k)^2 $), which is impossible unless $k = 1$.

As a result, except when $n = 3$ and $T(E_{11}) = D$ is invertible,
we see that  $T$ sends rank one matrices to rank one matrices.
If $n=3$ and $D$ is invertible and $\IF=\IC$, we must add one more equation to reach a contradiction:
\begin{equation}\label{eq:additional}
 (YD-\beta/2 I)+(YD-\beta /2 I)^\ast=0\hbox{ or equivalently } Y\in \bigl(\IA_3(\IF)+\IR I\bigr)D^{-1},
\end{equation}
which is obtained from~\eqref{eq:DY*}. Following the above arguments we see that $T$ maps a space $0\oplus M_{2}(\IC)$ whose real dimension is $8$ into a space
\begin{equation}\label{eq:intersection}
D^{-1}\bigl(\IA_3(\IF)+\IR I\bigr)\cap \bigl(\IA_3(\IF)+\IR I\bigr)D^{-1}.
\end{equation}
This intersection consists of matrices $X$ with $X=D^{-1}(i H+\lambda I) =(iK+\mu I)D^{-1}$, or equivalently,
$$iK+\mu I=D^{-1}(iH+\lambda I)D$$
for some skew-hermitian $iH,iK$ and real $
\lambda,\mu$.  Since skew-hermitian matrices have purely imaginary  eigenvalues, and the spectrum does not change under conjugation by $D$, we see that $\lambda=\mu$, and hence $K=D^{-1}HD$. Since $K,H,D$ are all hermitian, this forces $D^{-1}HD=(D^{-1}HD)^\ast= DHD^{-1}$, or equivalently, $H$ commutes with $D^2$. Recall that $D$ is diagonal and positive-definite with monotonically decreasing positive eigenvalues. If it is not a scalar matrix, then $D^2$ has at least two distinct eigenvalues,  in which case the real-linear subspace of such skew-hermitian $iH$ is contained in $i\IH_2\oplus i\IR$ (or in $i\IR\oplus i\IH_2$) with dimension $5$. In this case, the dimension of~\eqref{eq:intersection} is bounded above by $5+1=6$ which contradicts the fact that it contains the image of $0\oplus M_2(\IC)$ under a real-linear bijection $T$.

The only possibility left to consider is that $D=\mu_0 I\in\IR I$. Recall that $D$ was obtained by first modifying $T$ into $X\mapsto V_1T(U_1XU_2)V_2$ and then evaluating at $X=E_{11}=U_1^\ast A U_2^\ast$. Hence, the arguments so far show that $T$ maps rank-one either into rank-one or into a scalar multiple of some unitary matrix. Let $e_1,\dots,e_n$ be a standard basis of $\IF^n$.
Now, the first column $\IC^3 e_1^\ast$ is a real-linear subspace of dimension six which, besides $0$, consists only of rank-one matrices. A real-linear bijection $T$ maps it into a real-linear subspace of $M_3(\IC)\subseteq M_6(\IR)$. By Adam's fundamental result on dimensions of spaces of invertible matrices (\cite{Adams1962},  \cite{AdamsLaxPhillips1965}), in $M_6(\IR)$ (hence also in $M_3(\IC)$), the maximal possible dimension of a real-linear space which, besides $0$, consists of invertible matrices only, is $\rho(6)=2$, where $\rho$ denotes the Radon-Hurwitz function (c.f. also \cite{AdamsLaxPhillips1965} for a more general result valid for any size). Thus, there exists a nonzero vector $x\in\IC^3$ such that $T(xe_1^\ast)=uv^\ast$ is of rank-one. But then rank-one $E_{11}+\lambda xe_1^\ast=(e_1+\lambda x)e_1^\ast$ ($\lambda\in\IR$)  is mapped by (modified)  $T$ into
$$\mu_0 I + \lambda uv^\ast,$$
 which  is never of rank-one and which  also fails to be  a scalar multiple of a unitary  for  at least one real $\lambda$, yielding the desired contradiction. Indeed, if $\IF=\IC$, then $T$ leaves the set of rank-one matrices invariant.

If $n=3$ and $D$ is invertible and $\IF=\IR$
 we again  add~\eqref{eq:additional}.
Following the above arguments we see that $T$ maps a $4$-dimensional  space $0\oplus M_{2}(\IR)$ into a space
$D^{-1}\bigl(\IA_3(\IR)+\IR I\bigr)\cap \bigl(\IA_3(\IR)+\IR I\bigr)D^{-1}$. Since
$\dim D^{-1}\bigl(\IA_3(\IR)+\IR I\bigr)=4$ and $T$ is bijective we conclude that
$$D^{-1}\bigl(\IA_3(\IR)+\IR I\bigr)=\bigl(\IA_3(\IR)+\IR I\bigr)D^{-1}.$$
It follows that $D$ must be a scalar matrix. To see this, note that for every
$A\in \IA_3(\IR)$ there is $B\in\IA_3(\IR)$ and a scalar $\delta$ such that
$D^{-1}A=(B+\delta I)D^{-1}$. Multiplying by $D$ and
taking the trace gives $\delta=0$, and hence $B=D^{-1}AD\in \IA_3(\IR)$. By arbitrariness of $A\in \IA_3(\IR)$ and as $D$ is positive definite, it must indeed be a scalar matrix.

 In this case, $D=\gamma I$ and
 $T(0\oplus X)=Y\in D^{-1}\bigl(\IA_3(\IR)+\IR I\bigr) = \IA_3(\IR)+\IR I$ for every
 orthogonal $X\in U_{2}(\IR)$ and, moreover, $(D+ Y)(D+Y^T) =\alpha_1 I$ reduces to $YY^T\in\IR I$. Decompose
 $Y=\beta I +A$ for some skew-symmetric $A$ to see $YY^T=\beta^2 I -A^2=\beta^2 I +AA^T\in\IR I$, and since
 $3$-by-$3$ skew-symmetric matrices are never invertible we see that $A=0$.  This implies the bijection $T$ maps a
 basis of $0\oplus M_2(\IR)$ composed of matrices $0\oplus X$ with $X$
 orthogonal to $Y= T(0\oplus X)\in\IR I$, a
 contradiction.
\end{proof}

\begin{proof}[Proof of Theorem~\ref{thm:main1}]
The sufficiency is clear. To prove the necessity, Lemma~\ref{lem:rk-1-reduction} implies that~$T$ sends rank one matrices to rank one matrices.
The bijective preservers of rank-one  operators are known even if they are merely additive; on rank-one matrices they take one of the two forms
  (a) $xy^\ast\mapsto (M_\sigma x) (N_\sigma y^\ast)$ or (b) $xy^\ast\mapsto(M_\sigma y)(N_\sigma x^\ast)$  for some $\sigma$-linear bijections $M_{\sigma}\colon M_{n\times 1}(\IC)=\IC^n\to\IC^n$ and $N_{\sigma}:M_{1\times n}(\IC)\to M_{1\times n}(\IC)$, where $\sigma\colon\IC\to\IC$ is a field isomorphism; see e.g., Theorem 2.1 of \cite{Kuzma2002}. In our case, $T$ being real-linear implies $\sigma$ must fix every real number, and since it is a homomorphism of $\IC$, it is either the identity or complex conjugation. Define a linear bijection $M\colon\IC^n\to\IC^n$ on a standard basis by $Me_i:=M_{\sigma} e_i$ and likewise for $N$; then the forms (a)--(b) for $T$ read
    $$(a)\ A \mapsto MA^\sigma N^\ast\quad\hbox{or}\quad (b)\ A\mapsto M(A^\sigma)^T N^\ast$$ where $A^\sigma:=\bigl(\sigma(a_{ij})\bigr)_{ij}$ is obtained from $A$ by applying  $\sigma$ entrywise.

If $\sigma$ is complex conjugation  we replace $T$ by the map $X\mapsto T(\overline{X})$; the new $T$ will be linear, bijective, and will still leave the set $\IC\, {\mathcal U}_n(\IC)$ invariant.

It remains to show  $M,N$ are multiples of unitaries.  We may write
$M = X\, \diag(r_1, \dots, r_n)Y^*$ and $N = R\, \diag(s_1, \dots, s_n) S^*$ for some
$X, Y, R, S \in {\mathcal U}_n(\IF)$, $r_1 \ge \cdots \ge r_n > 0$ and
$s_1 \ge \cdots \ge s_n > 0$. Let $A = Y S^*$. Then $T(A)$ or $T(A^T)$ are multiples of unitaries, from which we deduce that $r_1s_1=r_ns_n$.  Thus $r_1 = r_n$ and $s_1= s_n$, so $M/r_1, N/s_1 \in {\mathcal U}_n(\IF)$ as desired.
\end{proof}

\begin{corollary}
 Let $n\ge 3$. A real-linear bijection $T\colon M_n(\IF)\to M_n(\IF)$ leaves the set $\IF\,{\mathcal U}_n(\IF)$ invariant if and only if there exists $r>0$  and a unitary $U$ such that $X\mapsto rUT(X)$ is an isomorphism or an anti-isomorphism of ${\mathcal U}_n(\IF)$.
\end{corollary}
\subsection{Matrix pairs with product attaining maximum norm value}

For notational simplicity, we shall occasionally  use
${\mathcal U}_n$ to represent ${\mathcal U}_n(\IF)$.
It is known that the spectral norm $\|\cdot\|$ is submultiplicative, i.e.,
$\|AB\| \le \|A\| \|B\|$ for any $A, B \in M_n(\IF)$.
We have the following result characterizing linear maps preserving matrix pairs attaining the equality.

\begin{theorem}\label{main2} Let $n \geq 3$.
A bijective real-linear map $T \colon M_n(\IF) \to M_n(\IF)$ satisfies
\[
\|T(A)T(B)\| = \|T(A)\| \, \|T(B)\|
\quad \text{whenever} \quad
\|AB\| = \|A\| \, \|B\|
\]
if and only if $T$ has the form $A \mapsto \gamma U A U^*$  or the form $A \mapsto \gamma U \overline{A} U^*$  for some nonzero $\gamma \in \IF$ and $U \in {\mathcal U}_n(\IF)$.
\end{theorem}

The sufficiency of the result is clear.  We will prove the converse after  the next two lemmas; the first one will be used also in the concluding subsections.

\begin{lemma}\label{L:basic}
	Let $A, B \in M_n(\IF)$ be nonzero.
	\begin{itemize}
		\item[{\rm (i)}] $\|AB\| = \|A\| \|B\|$ if and only if there exists a unit vector $x$ satisfying $B^* Bx = \|B\|^2 x$ and $A$ attains its norm at $Bx$.
		\item[{\rm (ii)}] $\|AB^*\| = \|A\| \|B^*\|$ if and only if there exists a unit vector $y$ that is norm-attaining for both $A$ and $B$.
	\end{itemize}
\end{lemma}

\begin{proof}
	For the first assertion, $\|AB\| = \|A\| \|B\|$ if and only if there exists a unit vector $x$ such that $B$ attains its norm at $x$ and $A$ attains its norm at $Bx$.  But $B$ attains its norm at $x$ if and only if $B^*Bx = \|B\|^2 x$.

	For the second assertion, $\|AB^*\| = \|A\| \|B^*\|$ if and only if there exists a unit vector $x$ such that $A$ attains its norm at $B^* x$ and $B^*$ attains its norm at $x$.  But, as follows from its SVD, $B$ attains its norm at $B^* x$.

	Conversely, if $y$ is norm-attaining for both $A$ and $B$, then $B^*By = \|B\|^2 y$.  Thus $By$ is norm-attaining for $B^*$, and its image under $B^*$ is norm-attaining for $A$.
\end{proof}

 \begin{lemma}\label{lam1} The following is equivalent for a matrix $A\in M_n(\IF)$.
 \begin{itemize}
  \item[{\rm (i)}] $\|AB\|=\|A\|\cdot\|B\|$ for every $B\in M_n(\IF)$
   \item[{\rm (ii)}] $\|AB^\ast\|=\|A\|\cdot\|B^\ast\|$ for every $B\in M_n(\IF)$
   \item[{\rm (iii)}] $\gamma A\in {\mathcal U}_n(\IF)$ for some $\gamma \in \IF$.
 \end{itemize}
 \end{lemma}
\begin{proof}
Claims (i) and (ii)  are clearly equivalent.

Suppose   (iii) holds with $A = \mu U$,   $U \in {\mathcal U}_n$ and $\mu\in\IF$. Then $\|AB\| = \|\mu UB\| = |\mu| \|UB\| = |\mu| \|B\| = \|A\| \|B\|$
 for any $B \in M_n(\IF)$, as claimed by (i).

 Conversely, suppose (i) holds but with a matrix $A$ which is  not a multiple of a matrix in~${\mathcal U}_n$.
 By the singular value decomposition, we see that $A = P (\sum_{j=1}^n s_j E_{jj} )Q$ with $P, Q \in {\mathcal U}_n$
 and $s_1 \ge \cdots \ge s_n \ge0$ such that $s_1 > s_n$. Let
 $B = Q^*E_{nn}$. Then $\|AB\| = s_n < s_1 =  \|A\|\cdot \|B\|$, a contradiction.
\end{proof}

\begin{proof}[Proof of Theorem~\ref{main2}]

 Suppose that $T$ is bijective real-linear such that
$$\|T(A)T(B)\| = \|T(A)\| \|T(B)\| \quad \hbox{ whenever } \quad \|AB\| = \|A\| \|B\|.$$
If $A \in {\mathcal U}_n$, then $\|AB\| = \|A\| \|B\|$ for all $B \in M_n(\IF)$. It follows that
 $\|T(A)T(B)\| = \|T(A)\| \|T(B)\|$. So, $T(A)$ is a multiple of a matrix in ${\mathcal U}_n$. By
 Theorem~\ref{thm:main1}, $T$ has the form
 $$A \mapsto  \gamma UAV \quad \hbox{ or } \quad A \mapsto  \gamma UA^TV \quad \hbox{ or } \quad A \mapsto  \gamma U\bar{A}V \quad \hbox{ or } \quad A \mapsto  \gamma U\bar{A}^TV$$
 for some   $\gamma \ne 0$ and $U, V \in {\mathcal U}_n$. We replace $T$ by the map
 $X \mapsto VT(X)V^*/\gamma$ (or, in the last two forms, by the map $X \mapsto VT(\bar{X})V^*/\gamma$; notice that $X\mapsto\bar{X}$ preserves norm-multiplicative pairs) and assume that $T$ has the form (a) $X \mapsto WX$, or (b) $X \mapsto WX^T$ with $W = VU$.
 Suppose $W$ is not a scalar matrix. If (a) holds, then there is a unit vector $x$ such that $|x^\ast Wx|<1$
 so that  we have $\|AB\| = \|A\| \|B\|$
 for $A=B=xx^\ast$, but
 $$\|T(A)T(B)\| = \|Wxx^*W xx^*\| = |x^*Wx| \|Wxx^*\| < \|Wxx^*\| \|W xx^*\|.$$
 Similarly, we can show that $W$ is a scalar matrix if (b) holds.  Thus, $T$ has the form $X \mapsto \mu UAU^*$
 or $X \mapsto \mu UA^T U^*$. Again, we may replace $T$ by $X \mapsto U^*T(X)U/\mu$ and assume that
 $T(X) = X$ for all $X$, or $T(X) = X^T$ for all $X$. Note that the latter case
 cannot hold because we can choose $A = E_{11}$ and $B = E_{12}$ so that $\|AB\| = \|A\| \|B\|$
 but $\|A^TB^T\| = 0  <  \|A^T\| \|B^T\|$.
\end{proof}

\subsection{Related results}
Clearly
$\|AB^*\| \le \|A\| \|B^*\| = \|A\| \|B\|$.
Using the same ideas to prove Theorem~\ref{main2}, we will also prove the following:

\begin{theorem}\label{thm:generalization_n>2} Let $\|\cdot\|$ be the spectral norm on $M_n(\IF)$, and let $T\colon  M_n(\IF) \rightarrow M_n(\IF)$,  $n\ge 3$, be a bijective real-linear map. The following are equivalent.
\begin{itemize}
\item[{\rm (i)}] $\|T(A)^* T(B)\| = \|T(A)\| \|T(B)\|$ whenever $\|A^*B\| = \|A\| \|B\|$.
\item[{\rm (ii)}] $\|T(A) T(B)^*\| = \|T(A)\| \|T(B)\|$ whenever $\|AB^*\| = \|A\| \|B\|$.
\item[{\rm (iii)}] There are $r > 0$, and $U, V \in {\mathcal U}_n(\IF)$ such that $T$ has the form
$A \mapsto rUAV$ or $A \mapsto rU\bar{A}V$.
\end{itemize}
\end{theorem}
Let us firstly directly show that (i) and (ii) are equivalent.
\begin{lemma}\label{L:equiv}
Let $S\colon M_n(\IF) \to M_n(\IF)$ be a real-linear surjective map.  Let $\Phi(X)=X^*$ and let $T=\Phi \circ S \circ \Phi$.  Then
$$\|A^*B\| = \|A^*\| \|B\| \implies \|S(A)^* S(B)\| = \|S(A)^*\| \|S(B)\| \quad \forall A,B \in M_n(\IF)$$
if and only if
$$\|CD^*\| = \|C\| \|D^*\| \implies \|T(C) T(D)^* \| = \|T(C)\| \|T(D)^* \|\quad \forall C,D \in M_n(\IF).$$
\end{lemma}

\begin{proof}
Notice that $\Phi$ is conjugate-linear, so $T$ is
real-linear. Now,
the first statement is equivalent to
\begin{equation*}
\|A^* (B^*)^*\| = \|A^*\| \|(B^*)^*\| \implies \|T(A^*) T(B^*)^*\| = \|T(A^*)\| \|T(B^*)^*\|.
\end{equation*}
Apply it with $C=A^*$ and $D=B^*$ to get the second statement.
\end{proof}
\begin{proof}[Proof of Theorem~\ref{thm:generalization_n>2}]
  By Lemma~\ref{L:equiv}  it suffices to prove that (ii) implies (iii). By Lemma~\ref{lam1}  such $T$ will leave the set $\IF\, {\mathcal U}_n$ invariant. Hence, $T(I)\in\IF\,{\mathcal U}_n$ and by replacing $T$ with $X\mapsto T(I)^{-1}T(X)$ we achieve that it is unital. The rest of the arguments follow  the proof of Theorem~\ref{main2}; we only need to show that the transposition map will not preserve norm multiplicative pairs in (ii). This follows by choosing
  $(A,B)=(E_{11},E_{21})$.
\end{proof}

\section{Real-linear preservers on  $2$-by-$2$ matrices}
In many preserver problems on matrices or operators,
the case of $2$-by-$2$ matrices is special because there are many special subgroups
in the semigroup of linear maps on 2-by-2 matrices.  This also happens to our problems.
 For one, the real-linear bijective preservers of $\IF\,{\mathcal U}_2$ have a richer structure than with $\IF\,{\mathcal U}_n$, $n\ge 3$, so the proofs will have to use different arguments.    Besides, as we show in Theorem~\ref{n=2} below, the preservers of  norm multiplicative pairs on real $2$-by-$2$  matrices also have richer structure than for $n\ge 3$.

\subsection{Preservers of multiples of unitary/orthogonal matrices}
When $n = 2$, we note that the eight matrices
\begin{equation}\label{eq:quaternions}
 I,\quad \Sigma_1 =\left[\begin{matrix}0 & i\\
                             i &0
                            \end{matrix}\right],\quad
 \Sigma _2  =\left[\begin{matrix}0 &1\\
                             -1&0
                            \end{matrix}\right],\quad\Sigma_3=\left[\begin{matrix}i&0\\
                             0 &-i
                            \end{matrix}\right] , \quad i I,\quad  i\Sigma_1,\quad i\Sigma_2,\quad i\Sigma_3
\end{equation}
are all unitary and form a real-linear basis for $M_2(\IC)$. Also, $\Sigma_1,\Sigma_2,\Sigma_3$  obey the same multiplication rules as the standard quaternionic units and thus
$$\bV_3:=\span_{\IR}\{I,\Sigma_1,\Sigma_2,\Sigma_3\}$$
is (isometrically isomorphic to) a  quaternion division algebra.
                            We have the following.

\begin{theorem}\label{P:unitary}
A bijective real-linear map $T\colon  M_2(\IF) \rightarrow M_2(\IF)$
maps matrices in ${\mathcal U}_2(\IF)$ to multiples of matrices in ${\mathcal U}_2(\IF)$ if and only if $T(I)$ is a nonzero multiple of a  matrix  in ${\mathcal U}_2(\IF)$ and the map $L$ defined by $L(X) = T(I)^{-1} T(X)$ satisfies one of the following.
\begin{itemize}
\item[{\rm (i)}] $\IF = \IR$ and $L(\bV_j) = \bV_j$ for $j = 1,2$, where
$\bV_1 = \span\{I_2, E_{12}-E_{21}\}$ and $\bV_2 = \span\{E_{11}-E_{22}, E_{12} + E_{21}\}$.
\item[{\rm (ii)}] $\IF=\IC$ and there are $U,V\in {\mathcal U}_2(\IC)$ and $\mu\in\IC\setminus\IR$ such that the unital map $L'\colon X \mapsto
 VL(UXU^\ast)V^\ast$ satisfies $L'(i I)=\mu I$ and
               $$L'(\Sigma_j) = a_j \Sigma_j +  b_j I_2\quad \hbox{and}\quad L'(i\Sigma_j)=\mu L'(\Sigma_j)$$     for some
real $a_1, a_2, a_3, b_1, b_2, b_3$ with  $a_1\ge a_2\ge| a_3|>0$. If $T$ is complex-linear, then $\mu=i$.
\end{itemize}
\end{theorem}

\begin{remark}
When $n \ge 3$, a linear bijection sending ${\mathcal U}_n(\IF)$ to
$\IF {\mathcal U}_n(\IF)$ will map rank-1 matrices to rank-1 matricees. This is not the case when $n = 2$.
The map $L$  in Theorem~\ref{P:unitary}
 could send rank-1 matrices to rank-2 matrices.  For example, consider a linear map  $L$ which fixes $I$ and $E_{12}-E_{21}$, and swaps $E_{11}-E_{22}$ with $2(E_{12}+E_{21})$. Then $L$ leaves the sets $\bV_1,\bV_2$, and $\bV_3=\span_{\IR}\{I,\Sigma_1,\Sigma_2,\Sigma_3\}$ invariant   but maps $E_{11}$ to the rank-2 matrix $\frac{1}{2}I + E_{12}+E_{21}$.
\end{remark}

The  proof  will rely on the following result.
\begin{lemma}\label{lem:reduction-2-by-2}
Let $\Sigma_j$ be as in~\eqref{eq:quaternions}. Suppose a real-linear and unital $L\colon M_2(\IC)\to M_2(\IC)$  leaves $\bV_3=\span_{\IR}\{I,\Sigma_1,\Sigma_2,\Sigma_3\}$ invariant. Then there are
unitary $U, V \in M_2(\IC)$ such that $L'\colon M_2(\IC)\to M_2(\IC)$, defined by
\[
L'(X) = V L(U X U^*)\,V^*
\]
satisfies
\[
L'(I) = I,\quad L'(\Sigma_{1}) = a_1 \Sigma_1 + b_1 I,\quad L'(\Sigma_{2}) = a_2 \Sigma_2 +  b_2I,\quad L'(\Sigma_3) = a_3 \Sigma_3 + b_3 I,
\]
where $a_k,b_k$ are real with $a_1\ge a_2\ge |a_3|$.
\end{lemma}

\begin{proof}
Since $L(I) = I$ and $L(\bV_3) \subseteq \bV_3$, if we use the orthonormal
basis of $\bV_3$ (with respect to standard inner product $\langle A,B\rangle:=\tr B^\ast A$)
\[
{\bf B} = \left\{\frac{I}{\sqrt{2}},\ \frac{\Sigma_1}{\sqrt{2}},\ \frac{\Sigma_2}{\sqrt{2}},\ \frac{\Sigma_3}{\sqrt{2}}\right\},\]
then the matrix of transformation  $L$ with respect to the basis $\bf B$ has the form
\[
R =
\begin{bmatrix}
1 &  w \\
0 & S
\end{bmatrix}
\]
where $w = (w_1, w_2, w_3)$ is a real vector and $S$ is a real matrix.
Let $P, Q \in \mathrm{SO}(3)$ be such that
\[
P S Q = \operatorname{diag}(a_1, a_2, a_3),
\]
where
$a_1 \ge a_2 \ge |a_3|$ are the singular values of $S$.
Let $\hat{P} = [1] \oplus P$ and $\hat{Q} = [1] \oplus Q$. Then
\[
\hat{P} \, R \, \hat{Q} =
\begin{bmatrix}
1 &  b \\
0 & D
\end{bmatrix}
\]
with $b = w Q = (b_1, b_2, b_3)$ and $D = \operatorname{diag}(a_1, a_2, a_3)$.
Here, $a_1 a_2 a_3 = \det(S)$.

By the classical correspondence of Lie groups,  $\mathrm{U}(2) / \mathrm{U}(1)  \cong \mathrm{SO}(3)$, where $\mathrm{U}(1)$ is considered  as a normal subgroup of $\mathrm{U}(2)$ consisting of all unimodular multiples of $I$. In fact, this correspondence takes $U\in \mathrm{U}(2)$ into (an orthogonal) matrix representing the  transformation $X\mapsto UXU^\ast$, restricted to an invariant subspace $\span_{\IR}\{\Sigma_1,\Sigma_2,\Sigma_3\}\subseteq\bV_3$ and relative to basis ${\bf B}\setminus\{I/\sqrt{2}\}$; see, e.g., \cite[Exercise B.4]{AubrunSzarek2017} or \cite[p.~54]{CarterSegalMacdonald};
elementary arguments were given in~\cite[Proposition 4.3]{Li-Tsai-Wang-Wong} (with correction $P_{W_1} = [1] \oplus\left[
\begin{smallmatrix}
\cos\varphi_1 & -\sin\varphi_1 \\
\sin\varphi_1 & \cos\varphi_1
\end{smallmatrix}\right]$).
Hence, we see that the
linear maps $L_1$ and $L_2$
 with $\hat{Q}$ and $\hat{P}$ as the operator matrices with respect to $\bf B$
correspond to the linear maps
\[
L_1(A) = U A U^* \quad\text{and}\quad L_2(A) = V A V^*
\]
for some $U, V \in \mathrm{U}(2)$. The conclusion follows.
\end{proof}

We first present a simplified proof of Theorem~\ref{P:unitary} for linear maps. Notice that this completely proves the case of real matrices.

\begin{proof}[Partial proof of Theorem~\ref{P:unitary} for linear maps] Suppose $T$ is a linear bijection such that $T({\mathcal U}_2)\subseteq \IF\,{\mathcal U}_2$.
We may replace it by  $X \mapsto T(I)^{-1}T(X)$ and assume that $T$ is unital.

First observe that a matrix $A \in M_2(\IF)$ is a scalar multiple of a  matrix in ${\mathcal U}_2$ if and only if $A$ is normal with two eigenvalues of equal modulus.  If $\IF = \IR$, this occurs when either $A \in \IR I$ (it has a repeated eigenvalue), $A$ is a trace-zero symmetric matrix (it has distinct real eigenvalues), or $A$ is the sum of a real scalar multiple of $I$ and a skew-symmetric matrix (when the eigenvalues are complex conjugates).  Thus
\begin{equation}\label{RU_2}
	\IR\,{\mathcal U}_2(\IR) =  (\IR I + \IA_2(\IR)) \cup H_2^0(\IR) = \textbf{V}_1 \cup \textbf{V}_2,
\end{equation}
where $H_2^0(\IR)$ denotes the real space of $2 \times 2$ symmetric matrices with trace zero.

If $\IF = \IC$, after a rotation one may assume that the eigenvalues of $A$ are  complex conjugates.  Then $A$ is scalar multiple of a unitary if and only if $A = e^{i\theta} (rI + K)$ for some $\theta, r \in \IR$ and $K \in \IA_2^0(\IC)$, where $\IA_2^0(\IC)$ is the real space of skew-hermitian $2 \times 2$ complex matrices with trace zero.  Thus
\begin{equation}\label{CU_2}
	\IC{\mathcal U}_2(\IC) = \IC (\IR I + \IA_2^0(\IC)) =\IC\,\span_{\IR}\{I,\Sigma_1,\Sigma_2,\Sigma_3\}=\IC\bV_3.
\end{equation}

Now suppose $T\colon M_2(\IF) \to M_2(\IF)$ is a bijective linear map satisfying $T({\mathcal U}_2) \subseteq \IF\,{\mathcal U}_2$.  For the case $\IF = \IR$, because $M_2(\IR) = \textbf{V}_1 \oplus \textbf{V}_2$, $\dim \textbf{V}_1 = \dim \textbf{V}_2 = 2$, $T(I)=I$, and $T$ maps $\textbf{V}_1 \cup \textbf{V}_2$ into itself, we must have
$T(\textbf{V}_j) = \textbf{V}_j$ for $j=1,2$
 as claimed by (i).  Conversely, if (i) holds, then $T$ will leave $\IR  {\mathrm O}_2= {\bf V}_1\cup\bV_2$ invariant. Thus the first assertion holds.

For the case $\IF = \IC$, observe that, for a nonscalar matrix $X$, $X + tI \in \IC\, {\mathcal U}_2$ for all $t \in \IR$ if and only if $X \in \IR I + \IA_2^0(\IC)$.
Since $T$ is unital, bijective,  and also satisfies $T(\IC\, {\mathcal U}_2) \subseteq \IC\, {\mathcal U}_2$,
 we have  $T(\bV_3)=\bV_3$ and item (ii) follows by Lemma~\ref{lem:reduction-2-by-2}. The converse will be proven below.
\end{proof}

\begin{proof}[Proof of Theorem~\ref{P:unitary}] We only need to consider the complex matrices. Suppose $T$ is a (real)linear bijection such that $T({\mathcal U}_2)\subseteq \IC\,{\mathcal U}_2$.
We may replace it by  $X \mapsto T(I)^{-1}T(X)$ and assume that $T$ is unital.
Recall that
$$\bV_3= \{rI + K : r \in \IR, K \in M_2(\IC), \tr K = 0, K^* = -K \}=\span_{\IR}\{I,\Sigma_1,\Sigma_2,\Sigma_3\}.$$
Recall
that $X+tI$
is a multiple of a unitary for all real numbers $t$ if and only if
$X = \alpha I + iH$ for a hermitian matrix $H$   which is of trace-zero or belongs to $\IR I$, and where $\alpha\in\IR$.
This is equivalent to  $X\in\bV_3\cup\IC I$ (a union of two real-linear subspaces). In particular,
since $\bV_3\supseteq\IR I$,
 the (real)linear unital  bijection  $T$  will map the set $\bV_3\cup\IC I$ into itself.
Since  it cannot map the four-dimensional real vector space $\bV_3$ into the two-dimensional real space $\IC I$, it
will  fix $\IR I$  and  will send $\bV_3$ onto itself and likewise will send $\IC I$ onto itself.

By applying the above arguments,
$$T(iI)=\mu I$$ for some nonreal complex number $\mu$.
Let us introduce the auxiliary real-linear bijection $T'\colon X\mapsto T(iI)^{-1}T(X)$ which maps $iI$ into the identity and maps  trace-zero hermitian  $H\in i\bV_3$ into a matrix $T'(H)$ such that $I+t T'(H)=T'(iI+tH)\in T'(i\bV_3)$ is a multiple of unitary for each real $t$. As shown above, this implies that $T'(H)\in\bV_3\cup\IC I$. Then, $T(H)=T(iI)T'(H)\in T(iI)(\bV_3\cup\IC I)=\mu \bV_3\cup\IC I$ and so $T(i\bV_3\cup\IC I)\subseteq\mu\bV_3\cup \IC I$. Again the four-dimensional real space $i\bV_3$ cannot be mapped into the two-dimensional real space $\IC I$, so
\begin{equation}\label{eq:T(iV_3)}
 T(i\bV_3)=\mu \bV_3.
\end{equation}
Recall
that $\bV_3$ is isomorphic to quaternions with  matrices
\[
\Sigma_1 =i(E_{12} + E_{21}) ,\quad \Sigma_2 = (E_{12} - E_{21}),\quad \Sigma_3 = i(E_{11} - E_{22})
\]
serving as quaternionic units. Take then any nonzero
$X\in\bV_3$ and let $U:=T(X)\in\bV_3$ and $V:=T(iX)\in\mu\bV_3$. It follows that
$$U(I+t U^{-1} V)=T((1+it)X)$$
is a multiple of a unitary for every real $t$, and so $U^{-1} V\in\bV_3\cup\IC I$. According to~\eqref{eq:T(iV_3)}, $V=\mu V'$ for some $V'\in\bV_3$, so that
\begin{equation}\label{eq:muUV'}
 \mu U^{-1} V'=U^{-1} V\in\bV_3\cup\IC I.
\end{equation}
Being isomorphic to quaternions, $\bV_3$ is closed under multiplication. Then, $U^{-1} V'\in\bV_3$, and since $\mu\in\IC\setminus\IR$ (i.e., $\bar{\mu}-\mu\neq0$)  we get from~\eqref{eq:muUV'} that $U^{-1} V\in({\bf V}_3\cup\IC I)\cap \mu{\bf V}_3=\IC I$. This, in combination with~\eqref{eq:T(iV_3)}, shows that
$$T(iX)=V=\lambda U=\lambda T(X);\qquad \lambda=\lambda_X\in\mu\IR\setminus\{0\}.$$
It is a standard result with real-linear bijections that the scalar $\lambda$ cannot depend on the choice of $X\in\bV_3$, and since $T(iI)=\mu I$ we get
$$T(iX)=\mu T(X);\qquad X\in i\bV_3,$$
and the form~(ii) again follows from Lemma~\ref{lem:reduction-2-by-2}.

To prove the converse let us first show that
\begin{equation}\label{eq:unitary-in-2-by-2}
A\in M_2(\IC)\hbox{ is a multiple of a unitary \ \ \ \  if and only if \ \  \ \ } A\in\alpha\bV_3\hbox{ for some scalar $\alpha$.}
\end{equation} In fact, a trace-zero normal $A$ is unitarily  similar to a diagonal matrix $\alpha\left[\begin{matrix}
  i &0\\
  0& -i                                                                                                                                                                     \end{matrix}\right]\in\alpha\bV_3$ where $\pm i\alpha$ are its eigenvalues, while a normal $A$ with nonzero trace and unimodular eigenvalues $\lambda_1=e^{i\theta_1}$ and $\lambda_2=e^{i\theta_2}$ is unitarily similar to
  $$\tfrac{\lambda_1+\lambda_2}{2}\left(I+\tfrac{\lambda_1-\lambda_2}{i(\lambda_1+\lambda_2)}\left[\begin{matrix}
  i &0\\
  0& -i                                                                                                                                                                     \end{matrix}\right]\right)=\tfrac{\lambda_1+\lambda_2}{2}\left(I+(\tan\tfrac{\theta_1-\theta_2}{2})\left[\begin{matrix}
  i &0\\
  0& -i                                                                                                                                                                     \end{matrix}\right]\right)\in(\tfrac{\lambda_1+\lambda_2}{2})\bV_3,$$ and the claim follows since $\bV_3$ is clearly closed under unitary similarity.

  It then follows that every unitary  $A$ is mapped by a real-linear unital $L'$
   into
  \begin{equation*}
   \begin{aligned}
   L'(A)&=L'(\mathrm{Re}(\alpha)X+\mathrm{Im}(\alpha)iX)=\mathrm{Re}(\alpha)L'(X)+\mathrm{Im}(\alpha)\mu L'(X)\\
   &=\bigl(\mathrm{Re}(\alpha)+\mathrm{Im}(\alpha)\mu)L'(X)\in\bigl(\mathrm{Re}(\alpha)+\mathrm{Im}(\alpha)\mu)\bV_3;\qquad X:=\tfrac{1}{\alpha}A\in\bV_3
  \end{aligned}
  \end{equation*}
so $L'$ maps unitary matrices into scalar multiples of unitaries.
\end{proof}

\subsection{Matrix pairs with product attaining maximum norm value}
The main result here is the following.
\begin{theorem}\label{n=2}
A bijective real-linear map $T\colon M_2(\IF) \rightarrow M_2(\IF)$ satisfies
$$\|T(A)T(B)\| = \|T(A)\| \|T(B)\| \qquad \hbox{ whenever } \qquad \|AB\| = \|A\| \|B\|$$
if and only if one of the following holds.
\begin{itemize}
\item[{\rm (a)}] $\IF = \IR$,
there exist a nonzero $\gamma \in \IR$, $c>0$, and $U \in {\mathcal U}_2(\IR)$ such that $T(A) = \gamma U \Phi_c(A) U^*$ for all $A \in M_2(\IR)$.
Here $\Phi_c\colon  M_2(\IR) \to M_2(\IR)$ is a linear bijection  defined by
$$\Phi_c \left( \begin{bmatrix} a & b \\ -b & a \end{bmatrix} \right) = \begin{bmatrix} a & b \\ -b & a \end{bmatrix}, \quad \Phi_c \left( \begin{bmatrix} a & b \\ b & -a \end{bmatrix} \right) = c \begin{bmatrix} a & b \\ b & -a \end{bmatrix}$$
for all $a,b \in \IR$.
\item[{\rm (b)}] $\IF = \IC$,
 there exist  $U,V\in {\mathcal U}_2(\IC)$, $\gamma > 0$
 and $\mu\in \IC\setminus \IR $
 such that the real-linear map $L\colon X\mapsto \gamma VT(UXU^\ast)V^\ast$ has the form
$ (X+iY) \mapsto X + \mu Y$  for  $X, Y \in \bV_3$. If $T$ is complex linear, then $\mu =
i$.
\end{itemize}
\end{theorem}

It is interesting to note that the linear preservers of norm multiplicative pairs on complex $2$-by-$2$ matrices behave as with $n\ge3$.

We need some technical lemmas to prove this result.
We first show that preservers of norm-multiplicative pairs will send $\IF I$ to $\IF I$.

\begin{lemma}\label{L:unital}
 Suppose a bijective real-linear  map $T\colon M_2(\IF) \rightarrow M_2(\IF)$ satisfies
$$\|T(A)T(B)\| = \|T(A)\| \|T(B)\| \qquad \hbox{ whenever } \qquad \|AB\| = \|A\| \|B\|.$$
Then  $T(\IF I) = \IF I$.
\end{lemma}
\begin{proof}
  Note that
$T$ preserves
$$\cN = \{ A \in M_2(\IF) : \|A\|^2 = \|A^2\|\}$$
which, by  \cite[Theorem~2.1]{Ptak} and \cite{GZ},
 is the set of normal matrices.
For $A \in \cN$, let
$$P(A) = \{B \in \cN: A +  B \in \cN \}.$$
We claim that   $\Span_{\IR} P(A) = M_2(\IF)$ if and only if $A \in\IF  I$. In fact, if $A\in\IF I$, then $P(A)$ is the set of all normal matrices and its span is $M_2(\IF)$ (namely, hermitian and skew-hermitian matrices are normal). Conversely,  if $A$ is not a scalar matrix, then $N\in P(A)$ if and only if $N\in\cN$ is normal and $(A+ N)^\ast (A+ N)=(A+ N) (A+ N)^\ast$. Expanding and rearranging gives that the normal matrices $A,N$ must satisfy
 \begin{equation}\label{eq:skew}
NA^\ast-A^\ast N = -(NA^\ast-A^\ast N)^\ast.
\end{equation}
Assume $\IF=\IC$. With no loss of generality,  let
$A=\diag(\lambda,\mu)$ be diagonal with distinct eigenvalues. Then~\eqref{eq:skew} is equivalent to a (normal) matrix $N$ satisfying $n_{21}=\overline{n_{12}} \frac{(\lambda-\mu)^2}{|\lambda-\mu|^2}$. The real span of such matrices $N$ is contained in the kernel of the real-linear functional $X=(x_{ij})\mapsto x_{21}-\overline{x_{12}}\frac{(\lambda-\mu)^2}{|\lambda-\mu|^2}$ and so is a proper real-linear subspace of $M_2(\IC)$.

Assume $\IF=\IR$.  Here,~\eqref{eq:skew} is equivalent to the fact that $NA^\ast-A^\ast N$ is a real skew-hermitian matrix in the range  of the elementary operator ${\bf T}_{A^\ast}\colon X\mapsto XA^\ast-A^\ast X$. Since $A^\ast$ is normal with  distinct eigenvalues, $\ker {\bf T}_{A^\ast} = \span\{I,A^\ast\}$ has dimension two.  By the Rank Theorem, the range of $T$ has dimension two.  On the other hand, the set of $2$-by-$2$ real skew-hermitian matrices is the one-dimensional space spanned by $J=E_{12}-E_{21}$.
 Thus $P(A)$ consists of normal matrices which ${\bf T}_{A^\ast}$ maps into $\IR J$, or equivalently, of normal matrices inside ${\bf T}_{ A^\ast}^{-1}(\IR J)$, which is an at most three-dimensional, and hence proper, subspace of $M_2(\IR)$.

Now if, for some $c\in\IF$,  $T(cI) = A$ is not a multiple of $I$, then
$$M_2(\IF) = T(M_2(\IF)) = T(\Span P(cI)) = \Span T(P(cI)) \subseteq \Span P(T(cI)) = \Span P(A)$$ is a proper subspace of $M_2(\IF)$,
a contradiction.
\end{proof}

We consider the case $\IF = \IR$ first.

\begin{lemma}\label{L:special-map}
	Fix $c>0$.  Define a linear map $\Phi_c\colon  M_2(\IR) \to M_2(\IR)$ as in Theorem~\ref{n=2}{\rm (a)}.
	 	Then
	$$\|\Phi_c(A) \Phi_c(B) \| = \|\Phi_c(A)\| \|\Phi_c(B)\| \; \text{ whenever } \; \|AB\| = \|A\| \|B\|,$$
	and
	$$\|\Phi_c(A) \Phi_c(B)^* \| = \|\Phi_c(A)\| \|\Phi_c(B)^* \| \; \text{ whenever } \; \|AB^*\| = \|A\| \|B^*\|.$$
\end{lemma}

\begin{proof}
Notice that $\Phi_c$ is the identity on $\bV_1=\span \{I_2, E_{12}-E_{21}\}$  and a scaling  by~$c$ on  $\bV_2=\span\{E_{11}-E_{22}, E_{12}+E_{21}\}$  ($\bV_1,\bV_2$ are as in Theorem~\ref{P:unitary}).
Since $\IR\,{\mathcal U}_2(\IR)=\bV_1\cup\bV_2$,  $\Phi_c$ maps orthogonal matrices to scalar multiples of orthogonal matrices, so if either $A$ or $B$ is a multiple of an orthogonal matrix, the assertions hold.  So we henceforth assume that neither $A$ nor $B$ is a multiple of an orthogonal matrix, i.e., they each have distinct singular values.

	\textbf{Step 1:} We start by proving that if $B$ attains its norm at $v$, then so does $\Phi_c(B)$.

	Relative to $M_2(\IR)=\textbf{V}_1\oplus \textbf{V}_2$
	 	we may decompose
	\[
	B = \alpha Q + \beta U,
	\]
	where \( Q \in \mathrm{SO}(2) \) is a rotation and  $U\in \mathrm{O}(2)$ is a reflection.  	Because $B$ is not a multiple of an orthogonal matrix, $\alpha \beta \ne 0$.  Notice that
	\[
	B^T B = (\alpha Q + \beta U)^T (\alpha Q + \beta U) = (\alpha^2 + \beta^2) I +  \alpha \beta (Q^TU+U^T Q).
	\]
	But \( Q^T U \in \mathrm{O}(2) \) and \( \det(Q^T U) = -1 \), so \( Q^T U \) is a reflection, hence an involution. Therefore, $(Q^T U)(Q^T U)=I$ which gives
	\[ Q^T U = U^T Q
	\]
	and  then
	\begin{equation}\label{eq:B-and-T(B)1}
		B^T B = (\alpha^2 + \beta^2) I + 2 \alpha \beta U^TQ.
	\end{equation}
	Now,
	\( B \) attains the norm at \( v \) if and only if  \( B^T B v = \|B\|^2 v \) or, equivalently, \( v \) is an eigenvector of \(2\alpha\beta U^T Q \), corresponding to the largest eigenvalue. Notice that
	\[ \Phi_c(B) = \alpha Q + c \beta U ,\] so replacing \( \beta \) with \( c\beta \) in~\eqref{eq:B-and-T(B)1}, and keeping in mind that $c>0$, shows that \( \Phi_c(B) \) also attains its norm at \( v \).  This proves Step 1.

	Consequently, if $A$ and $B$ both attain their norms at $v$, then so do $\Phi_c(A)$ and $\Phi_c(B)$.  The second assertion of this lemma follows by Lemma~\ref{L:basic}.
	\smallskip

	\textbf{Step 2:} To prove the first assertion of this lemma we now show that if $B$, $\Phi_c(B)$ attain their norms at $v$, then $Bv$ and $\Phi_c(B)v$ are parallel.

	Recall that $v$ is an eigenvector of $W=Q^TU=U^T Q$, say to eigenvalue~$\lambda > 0$, and write
	\[
	Bv = (\alpha Q + \beta Q Q^T U)v = Q(\alpha I + \beta W)v= Q(\alpha v + \beta \lambda v) = (\alpha + \beta \lambda) Qv,
	\]
	Then
	\[
	\Phi_c(B)v = Q(\alpha I + c \beta W)v = (\alpha + c \beta \lambda) Qv.
	\]
	is clearly parallel to $Bv$.

	Now suppose $B$ attains its norm at a unit vector $v$ and $A$ attains its norm at $Bv$.  By Step 1, $\Phi_c(B)$ attains its norm at $v$ and $\Phi_c(A)$ attains its norm at $Bv$, which is parallel to $\Phi_c(B)v$ by Step~2, so the first assertion of this lemma follows from Lemma~\ref{L:basic}.
\end{proof}

\begin{proposition}\label{P:n=2,real}
	Suppose $T\colon M_2(\IR) \rightarrow M_2(\IR)$ is a unital linear bijection satisfying either of the identities
	\begin{itemize}
		\item[{\rm (i)}] $\|T(A)T(B)\| = \|T(A)\| \|T(B)\|$ whenever $\|AB\| = \|A\| \|B\|$, or the identity
		\item[{\rm (ii)}] $\|T(A)T(B)^*\| = \|T(A)\| \|T(B)^*\|$ whenever $\|AB^*\| = \|A\| \|B^*\|$.
	\end{itemize}
	Then there exists $c>0$ and $U \in {\mathcal U}_2(\IR)$ such that $T(A) = U \Phi_c(A) U^*$ for all $A \in M_2(\IR)$, where $\Phi_c\colon M_2(\IR) \rightarrow M_2(\IR)$ is as in Theorem~\ref{n=2} {\rm (a)}.
 \end{proposition}

\begin{proof}
	For both identities, by Lemma~\ref{lam1}, $T$ must map orthogonal  matrices (orthogonals for short) to multiples of orthogonals.  By
	Theorem~\ref{P:unitary}, and since $T$ is unital, $T(\bV_j) = \bV_j$ for $j=1,2$.
	By conjugating with a suitable orthogonal matrix we may assume
	$$T \left(\begin{bmatrix} 1 & 0 \\ 0 & -1 \end{bmatrix} \right) = c\begin{bmatrix} 1 & 0 \\ 0 & -1 \end{bmatrix}$$
	for some $c > 0$.  We may then compose $T$ with $\Phi_c^{-1}$ (which is a unital bijection satisfying both identities by Lemma~\ref{L:special-map}) and assume $c=1$, so $T$ fixes diagonal matrices.
	Moreover
	$$T \left(\begin{bmatrix} 0 & 1 \\ -1 & 0  \end{bmatrix} \right) = a \begin{bmatrix} 0 & 1 \\ -1 & 0  \end{bmatrix} + bI$$
	for some $a,b \in \IR$ (by further conjugating with $\left[ \begin{smallmatrix} 1 & 0 \\ 0 & -1 \end{smallmatrix} \right]$ if necessary, we may assume $a>0$),
	 	and
	$$T \left(\begin{bmatrix} 0 & 1 \\ 1 & 0  \end{bmatrix} \right) = d \begin{bmatrix} 0 & 1 \\ 1 & 0  \end{bmatrix} + e \begin{bmatrix} 1 & 0 \\ 0 & -1 \end{bmatrix}.$$
	for some $d,e \in \IR$.
	\bigskip

	For reference we note the following fact, which follows easily from SVD.

	\textbf{Observation:} Let $e,f$ be orthonormal vectors in $\IR^2$.  Let $X \in M_2(\IR)$.  Then
	\begin{equation}\label{E:NA}
		X \text{ attains its norm at } e \iff \\ X^TXe = \|X\|^2 e \iff Xe \perp Xf \text{ and } \|Xe\| \geq \|Xf\|.
	\end{equation}

	\textbf{Step 1:} We shall show that $d=a$ and $e=b$.

	Note $\|XE_{11}\| = \|X\| \|E_{11}\|$ if and only if $X$ attains its norm at $e_1$; by~\eqref{E:NA} this occurs if and only if the columns of $X$ are orthogonal, and the norm of the second column does not exceed the norm of the first column.  In this case we may write
	\begin{equation}\label{col-form}
		X = \begin{bmatrix} \alpha & t \beta \\ \beta & - t \alpha \end{bmatrix}
	\end{equation}
	for some $\alpha, \beta \in \IR$ and $t \in [-1,1]$.  Since $T(E_{11}) = E_{11}$ is symmetric, for both identities we must have $\|T(X) E_{11}\| = \|T(X)\| \|E_{11}\|$, so $T(X)$ must have the same form~\eqref{col-form} and its columns must be orthogonal.
	Computing, we have
	$$T(X) = \begin{bmatrix} \alpha + \frac{\beta}{2} (e-b+t(e+b)) & \frac{\beta}{2}(d-a+t(d+a)) \\ \frac{\beta}{2} (d+a+t(d-a)) & -t\alpha + \frac{\beta}{2}(-e-b+t(b-e)) \end{bmatrix}$$
		and the dot product of its columns equals
	$$\frac{\beta(1-t^2)}{2}(\alpha (d-a) - \beta (bd+ae)).$$
	Since this must equal zero for all $\alpha, \beta$, we must have $d=a$ and consequently $e=-b$.  Substituting these into our expression for $T$, we have
	$$T\left( \begin{bmatrix} 0 & 1 \\ 0 & 0 \end{bmatrix} \right) = \begin{bmatrix} 0 & a \\ 0 & b \end{bmatrix}, \quad T\left( \begin{bmatrix} 0 & 0 \\ 1 & 0 \end{bmatrix} \right) = \begin{bmatrix} -b & 0 \\ a & 0 \end{bmatrix}.$$

	\textbf{Step 2:} We shall show that $b=0$ and $a=1$.

	Observe that, for all $t \in \IR$, $\begin{bmatrix} 1 \\ 1 \end{bmatrix}$ is a norm-attaining vector for $X_t = \begin{bmatrix} 1 \\ t \end{bmatrix} \begin{bmatrix} 1 & 1 \end{bmatrix}$.  Let $C = X_1$.

	\begin{itemize}
		\item Suppose  $T$ satisfies the first identity.  Since $\|X_t C\| = \|X_t\| \|C\|$ we  must have  $\|T(X_t) T(C)\| = \|T(X_t)\| \|T(C)\|$.  Since $T(C) = \begin{bmatrix} 1-b & a \\ a & 1+b \end{bmatrix}$ attains its norm at a unique (up to sign) unit eigenvector $v$, $T(X_t)$ must attain its norm at $v$ for all $t \in \IR$.
		\item Suppose $T$ satisfies the second identity.
		Since $X_t$ and $C$ have a common norm-attaining vector for all $t \in \IR$, by Lemma~\ref{L:basic} so must $T(X_t)$ and $T(C)$.
	\end{itemize}

	Thus in either case, $T(X_t)$ attains its norm at a fixed vector $v$ for all $t \in \IR$.  Note
	$$T(X_t) = \begin{bmatrix} 1-bt & a \\ at & b+t \end{bmatrix} = P + tQ, \quad \text{ where } P = \begin{bmatrix} 1 & a \\ 0 & b \end{bmatrix}, \; Q = \begin{bmatrix} -b & 0 \\ a & 1 \end{bmatrix}.$$
	Since $P+tQ$ attains its norm at $v$, $(P+tQ)^T (P+tQ) v = \|P+tQ\|^2 v$, so $v$ is an eigenvector for $P^T P + t(Q^T P + P^TQ) + t^2 Q^T Q$ for all $t \in \IR$, and hence an eigenvector for
	$P^T P = \begin{bmatrix} 1 & a \\ a & a^2 + b^2 \end{bmatrix}$, $Q^T Q = \begin{bmatrix} b^2 + a^2 & a \\ a & 1 \end{bmatrix}$, and $Q^T P + P^T Q = 2b \begin{bmatrix} -1 & 0 \\ 0 & 1 \end{bmatrix}$.

	Since $a>0$, $P^T P$ is not diagonal, so the only way for $v$ to be an eigenvector for both $P^T P$ and $Q^T P + P^T Q$ is for $b=0$.  Then the only way for $v$ to be an eigenvector for both $P^T P$ and $Q^T Q$ is for $a=1$.
	Thus, $T$ reduces to the identity mapping and the proposition
	holds.
\end{proof}

\begin{proof}[Proof for the real case of  Theorem~\ref{n=2}] Sufficiency was proven in Lemma~\ref{L:special-map}.
To prove necessity, note that
$T$ must map orthogonals to multiples of orthogonals by Lemma~\ref{lam1}.
   By Lemma~\ref{L:unital} we may, after scaling $T$ if necessary, assume that $T$ is unital. The rest follows by Proposition~\ref{P:n=2,real}.
\end{proof}
\smallskip

We now turn our focus to the case when $\IF = \IC$.
Recall that  $\bV_3=\span_{\IR}\{I,\Sigma_1,\Sigma_2,\Sigma_3\}$.

\begin{lemma}\label{lem:aux2-by-2-reallinear}
 	Suppose $T\colon M_2(\IC) \rightarrow M_2(\IC)$ is a  unital real-linear bijection which satisfies any one among the identities
    \begin{itemize}
     \item[{\rm (i)}] $\|T(A)T(B)\| = \|T(A)\| \|T(B)\| \qquad \hbox{ whenever } \qquad \|AB\| = \|A\| \|B\|$.
     \item[{\rm (ii)}] $\|T(A)T(B)^\ast\| = \|T(A)\| \|T(B)^\ast\| \qquad \hbox{ whenever } \qquad \|AB^\ast\| = \|A\| \|B^\ast\|$.
      \end{itemize}
	Then there exist  unitary $U,V$ such that $L\colon X\mapsto  V T(U XU^\ast)V^\ast$ fixes all diagonal matrices in $\bV_3$ and multiplies off-diagonal ones in $\bV_3$  by some $w>0$.
\end{lemma}

\begin{proof}
    By Lemma~\ref{lam1}, $T$ maps unitary matrices into scalar multiples of unitary matrices.
Using the unital assumption, we may apply Theorem~\ref{P:unitary} (ii)
 and modify $T$ to the map $X \mapsto VT(UXU^*)V^*$ with $U, V \in U_2(\IC)$ and assume that

\begin{equation}\label{eq:AB*-on-2-by-2anew}
\begin{aligned}
  &T \left( \left[\begin{smallmatrix} i & 0 \\ 0 & -i \end{smallmatrix}\right] \right) = a \left[\begin{smallmatrix} i & 0 \\ 0 & -i \end{smallmatrix}\right] +bI,\quad T \left( \left[\begin{smallmatrix} 0 & i \\ i & 0 \end{smallmatrix}\right] \right) = w\left[\begin{smallmatrix} 0 & i \\ i & 0 \end{smallmatrix}\right] +cI, \quad T \left( \left[\begin{smallmatrix} 0 & 1 \\ -1 & 0 \end{smallmatrix}\right] \right) = z\left[\begin{smallmatrix} 0 & 1 \\ -1 & 0 \end{smallmatrix}\right] + dI\\
  &T(iX)=\mu T(X);\qquad X\in\bV_3
\end{aligned}
\end{equation}
for some $\mu\in\IC\setminus\IR$ and some real constants satisfying $a\ge w\ge|z| > 0$. Notice that, due to bijectivity and unitality, $a, w, z$  are nonzero, so $a,w>0$.

Let $A_r = \begin{bmatrix} (r+1)i& 0 \\ 0 & (r-1)i \end{bmatrix}$, $r>0$. By~\eqref{eq:AB*-on-2-by-2anew} and as $T$ is unital,
$$T(A_r)=(b+r\mu) I +a\begin{bmatrix} i &0 \\ 0 & -i \end{bmatrix} = \begin{bmatrix}
 i a +b +r\mu  & 0 \\
 0 & -i a +b +r\mu
 \end{bmatrix} $$
Write $\mu=|\mu|e^{i\theta}$ and assume first that $\sin\theta>0$. Since $a,r>0$ and $b\in\IR$, one computes that
$A_r=-A_r^\ast$, $T(A_r)$, and $T(A_r)^\ast$ all
 attain their norms only at scalar multiples of their eigenvector $e_1$ (their $(11)$-entry has larger modulus than their $(22)$-entry).  Notice that $\|XA_r\| = \|X \| \|A_r\|$  for all $r > 0$ if and only if $X$
attains its norm on the vector $e_1$.
 This is equivalent to its SVD equaling $X =x e_1^\ast+ye_2^\ast$ for some orthogonal $x,y$ with $\|x\|\ge \|y\|$,  that is, $X$ must have the form
\begin{equation}\label{e1norm1anew}
X = \begin{bmatrix} \alpha & t \bar{\beta} \\ \beta & -t\bar{\alpha} \end{bmatrix}
\end{equation}
for some $\alpha, \beta, t \in \IC$ with $|t| \leq 1$. Thus, assuming either (i) or (ii) in the statement of this lemma, $X$ from~\eqref{e1norm1anew} is mapped into $T(X)$ which must also attain its norm at $e_1$ and hence have the same form~\eqref{e1norm1anew}.

However if, with $\mu=|\mu|e^{i\theta}$, one has $\sin\theta<0$, then  one computes that $T(A_r)$ and $T(A_r)^*$ both attain their
norm only at scalar multiples of their eigenvector $e_2$. Proceeding as above one sees that, under either of the assumptions (i) or (ii), $X$ of the form~\eqref{e1norm1anew} is mapped into a matrix $T(X)$ which attains its norm at $e_2$, hence is of the same form  as in~\eqref{e1norm1anew}  except that $|t| >1$. Hence, regardless if $\sin\theta$ is positive or negative, $X$ of the form~\eqref{e1norm1anew} is mapped into a matrix $T(X)$ with  its first  column orthogonal to its  second one.
 If we set $\alpha=\cos\pi/4$ and $\beta = e^{i\theta}\sin\pi/4$ and $t=1/3$, then under assumptions (i)--(ii),
$$T(X) = \frac{1}{6 \sqrt{2}}
\begin{bmatrix}
 x & y \\
 u & v \\
\end{bmatrix}$$
where
$$x=-4 i a \mu -4 b \mu -(c+i d) e^{-i \theta } (2 \mu -i)-(c-i d) e^{i \theta } (2 \mu +i)+2,$$
$$y =e^{i \theta } (w-z) (1-2 i \mu
   )-i e^{-i \theta } (w+z) (2 \mu -i),$$
$$u=e^{i \theta } (w+z) (1-2 i \mu )+e^{-i \theta } (z - w) (1 + 2 i  \mu),$$
$$v=4 i a \mu -4 b \mu -(c+i d) e^{-i \theta } (2 \mu
   -i)-(c-i d) e^{i \theta } (2 \mu +i)+2.$$

Since  the columns of $T(X)$  must be orthogonal, so
\begin{equation}\label{orth-cols1anew}
	x \bar{y} + u \bar{v} = 0.
\end{equation}
The left-hand side of the above equation is a trigonometric polynomial in $e^{i\theta}$:
 $$-8(\mu-\mathrm{Re}\,\mu  )  \left(i e^{-i \theta } (a-i b+1) (w-z)-e^{i \theta } (b+i (a-1)) (w+z)-2 (c z+i d w)\right)=0.$$
Its coefficients therefore vanish identically and since, by the assumptions $\mu\notin\IR$, $a,w>0$, and all other variables are real with $z \ne 0$, the only possibility is
$$a=1,\quad b=0,\quad z=w>0,\quad d=c=0.$$
This shows that $T$ is the identity on diagonal matrices from $\bV_3$, while on off-diagonal matrices from $\bV_3$ it acts as
$$T(\Sigma_1)=w \Sigma_1,\quad T(\Sigma_2)=w\Sigma_2,$$
as claimed.
\end{proof}

\begin{proof}[Proof for the complex case of  Theorem~\ref{n=2}]
By Lemma~\ref{L:unital} we may, after scaling $T$ if necessary, assume that $T$ is unital.  Then by Lemma~\ref{lem:aux2-by-2-reallinear} we can pre- and post- compose it with a conjugation by suitable unitaries to achieve that it also fixes diagonal matrices in $\bV_3$ and multiplies off-diagonal ones in $\bV_3$ by some $w>0$. Notice that $\bV_3$ is invariant under unitary conjugation. Hence, we only need to show that $w=1$.

To this end, we use
the same arguments as in Step~2 of  the proof of Proposition~\ref{P:n=2,real}.
Let $B = \left[\begin{smallmatrix}
1  & 1\\
1& 1\end{smallmatrix}\right]$ and $A=
\left[\begin{smallmatrix}
1 & 1\\
te^{i\theta} & te^{i\theta}\end{smallmatrix}\right]$  ($t,\theta\in\IR$); their
images are
$T(B)= \left[
\begin{smallmatrix}
 1 & -i \mu  w \\
 -i \mu  w & 1 \\
\end{smallmatrix}
\right]$, which is a normal matrix that attains its norm only at scalar multiples of eigenvectors $(1,\pm1)^t$, and $T(A)= \frac{1}{2}\left[
\begin{smallmatrix}
 \left(-i \mu +e^{-i \theta } (1+i \mu ) t+1\right)
   & -i w \left(\mu +e^{-i \theta } (\mu -i)
   t+i\right) \\
 - i w \left(\mu +e^{i \theta } (\mu +i) t-i\right) &
    \left(i \mu +e^{i \theta } (1-i \mu ) t+1\right)
\end{smallmatrix}
\right]$.  Since $\|AB\| = \|A\| \|B\|$, $\|T(A) T(B)\| = \|T(A)\| \|T(B)\|$.  A calculation, as in Step~2 of  the proof of Proposition~\ref{P:n=2,real}, then
 shows that $w=1$. Hence, $T$ is identity on $\bV_3$ and, by Theorem~\ref{P:unitary}(ii), multiplication by $\frac{\mu}{i} $ on $i\bV_3$.
 \medskip

 To prove the converse we note first that,  as shown in~\eqref{eq:unitary-in-2-by-2}, for any unitary $U\in M_2(\IC)$ there exists a (unimodular) scalar $\alpha=\alpha_1+i\alpha_2$  such that $U=\alpha W$ for some $W\in\bV_3$. Then,
\begin{equation*}
 \begin{aligned}
T(U)&=\alpha_1 T(W)+\alpha_2 T(iW)=(\alpha_1+\mu\alpha_2)W=\frac{\alpha_1+\mu\alpha_2}{\alpha} U.\\
T(iU) &=-\alpha_2T(W)+\alpha_1T(iW)=\frac{-\alpha_2+\mu\alpha_1}{\alpha} U=\frac{-\alpha_2+\mu\alpha_1}{\alpha_1+\mu\alpha_2} T(U)
\end{aligned}
\end{equation*}
 Notice also that the space $\bV_3$, which is a real-linear span of the identity and trace-zero skew-hermitian matrices, is closed under unitary conjugation, and also closed under multiplication (since $\Sigma_i\in\bV_3$ act as quaternionic units). In particular, if $U,V$ are unitary such that $UV^\ast\in\alpha \bV_3$, then also $U\Sigma_3V^\ast=U\Sigma_3 U^\ast\cdot UV^\ast\in \bV_3\cdot(\alpha \bV_3)=\alpha\bV_3$. This shows that
\[
\begin{aligned}
T(UV^\ast) &= \frac{\alpha_1 + \mu \alpha_2}{\alpha}\, UV^\ast, &
T(U \Sigma_3 V^\ast) &= \frac{\alpha_1 + \mu \alpha_2}{\alpha}\, U \Sigma_3 V^\ast, \\
T(i UV^\ast) &= \frac{-\alpha_2 + \mu \alpha_1}{\alpha}\, UV^\ast, &
T(i U \Sigma_3 V^\ast) &= \frac{-\alpha_2 + \mu \alpha_1}{\alpha}\, U \Sigma_3 V^\ast.
\end{aligned}
\]
Now, one decomposes
$$\diag(1,te^{i\theta})=
 \tfrac{1}{2}\,\bigl(1 + t \cos\theta + t i \sin\theta \bigr)\,I_2
\;+\;
\tfrac{1}{2}\,\bigl(-i + t i \cos\theta - t \sin\theta \bigr)\,\Sigma_3
;\qquad 0\le t\le 1,\;0\le \theta\le 2\pi.$$
Then, $T$ maps an arbitrary normalized $A=UD V^\ast\in M_2(\IC)$, given in its SVD with $D=\diag(1,t)$ (here, $0\le t\le 1$), and with $\alpha=\alpha_1+i\alpha_2$ as above, into
\begin{equation*}
 \begin{aligned}
  T(A)&=T\bigl(\tfrac{1}{2}\,\bigl(1 + t\bigr) UV^\ast\bigr)+T\bigl(\tfrac{1}{2}\,\bigl(-i + t i  \bigr)U\Sigma_3U^\ast\cdot UV^\ast\bigr)\\
  &=\tfrac{\alpha_1+\mu\alpha_2}{2\alpha}\,U\Bigl(\bigl(1 + t  )I \Bigr)V^\ast+\tfrac{-\alpha_2+\mu\alpha_1}{2\alpha}\,U\Bigl( \bigl(-1+t  )\Sigma_3 \Bigr)V^\ast\\
  &= U\left(\frac{\left(\alpha _2-i \alpha _1\right) (\mu +i) }{2 \left(\alpha _1+i \alpha
   _2\right)}\left[
\begin{matrix}
 1 & 0 \\
 0 & t \\
\end{matrix}
\right]+\frac{\left(\alpha _2+i \alpha _1\right) (\mu -i)
   }{2 \left(\alpha _1+i \alpha _2\right)}\left[
\begin{matrix}
t & 0 \\
 0 & 1 \\
\end{matrix}
\right]\right)V^\ast
 \end{aligned}
\end{equation*}
                   Clearly, this is  \emph{a generalized SVD} of $T(A)$, (i.e., it takes the form $T(A) = UDV^\ast$ for unitary $U,V$ and diagonal $D=\frac{1}{2\alpha}\left[
\begin{matrix}
 i  (\mu -i) t \bar{\alpha}-i \alpha  (\mu +i) & 0 \\
 0 & i (\mu -i) \bar{\alpha}-i \alpha  (\mu +i) t \\
\end{matrix}
\right]$, which is no longer required to be positive semidefinite). One further computes that
$$|d_{11}|^2-|d_{22}|^2=(1-t^2)\mathrm{Im}(\mu).$$
and since $0\le t\le 1$ (and $\alpha\notin\IR$), then $|d_{11}|\ge |d_{22}|$ if and only if $\mathrm{Im}(\mu)>0$.

Notice that this condition is independent of $\alpha$ and $t$, and depends only on the signature of the imaginary part of $\mu$.
Moreover, since the generalized SVD of $A$ and of $T(A)$ share the same unitary matrices $U$ and $V$, we see that either $A$ and $T(A)$ attain their norm on the same vector $v=Ve_1$, and map it into the same vector $u=Av \propto Ue_1$ (modulo a scalar multiple) for every $A\in M_2(\IC)$ (when $\mathrm{Im}(\mu)>0$) or else  $T(A)$ attains its norm on a vector $v=Ve_2$, perpendicular to a vector $Ve_1$ where $A$ attains its norm, and  maps it into $Ue_2$, which is again perpendicular to $A(Ve_1)\propto Ue_1$ for every $A\in M_2(\IC)$ (when $\mathrm{Im}(\mu)<0$).

In both cases, (when $\mathrm{Im}(\mu)>0$ or when $\mathrm{Im}(\mu)<0$) we see that if $B$ attains its norm on $Av$, where $v$ is the vector where $A$ attains its norm, then a similar property  is true for $T(A)$ and $T(B)$ (perhaps with  vectors perpendicular to $v$ and to $Av$). But this is exactly equivalent to $T$ preserving norm-multiplicative pairs.
\end{proof}

\subsection{Related results}

By a slight adaptation of techniques we developed for preservers of norm-multiplicative pairs we can also classify the preservers of sesqui-norm-multiplicative pairs on  $2$-by-$2$ matrices.

\begin{theorem}\label{T:n=2}
Suppose $T\colon M_2(\IF) \to M_2(\IF)$ is a real-linear bijection. Then, the following are equivalent.
	\begin{itemize}
		\item[{\rm (a)}] $\|T(A)T(B)^*\| = \|T(A)\| \|T(B)^*\|$ whenever $\|AB^*\| = \|A\| \|B^*\|$.
		\item[{\rm (b)}] $\|T(A)^*T(B)\| = \|T(A)^*\| \|T(B)\|$ whenever $\|A^*B\| = \|A^*\| \|B\|$.
		\item[{\rm (c)}] The map $X \mapsto T(I)^{-1} T(X)$ satisfy  Theorem {\rm~\ref{n=2} (a)--(b)}.
	\end{itemize}
\end{theorem}

\begin{proof}
We start by showing that both (a) or (b) imply (c). By Lemmas~\ref{L:equiv} and~\ref{lam1}, such $T$ maps $U_2(\IF)$ into $\IF U_2(\IF)$, so $T(I)\in\IF\,{\mathcal U}_2(\IF)$.  Since the map $X \mapsto \gamma UX$ satisfies both (a) and (b) for all nonzero $\gamma \in \IF$ and $U \in U_2(\IF)$, we may replace $T$ with $X\mapsto T(I)^{-1}T(X)$ and thereby assume that $T$ is unital.
When $\IF = \IR$, the result follows from Proposition~\ref{P:n=2,real} and Lemma~\ref{L:equiv}.

When $\IF = \IC$ and (a) holds, we apply Theorem~\ref{P:unitary} and Lemma~\ref{lem:aux2-by-2-reallinear}.  One can then show that $w=1$ in Lemma~\ref{lem:aux2-by-2-reallinear} by using the same argument as that in the second paragraph of the proof the complex case of Theorem~\ref{n=2} (because $B^*=B$) so (c) holds.  By using Lemma~\ref{L:equiv}, (b) implies (c) also holds.

Conversely, in the real case, Lemmas \ref{L:special-map} and \ref{L:equiv}) show that (c) implies both (a) and (b).  In the complex case, (c) implies (a) by using Lemma \ref{L:basic} and arguments similar to those in the proof of Theorem~\ref{n=2}, complex case.  As before, one then obtains (c) implies (b) by using Lemma~\ref{L:equiv}.
\end{proof}

\medskip\noindent
{\bf \large Acknowledgment}

This research was supported in part by the Slovenian Research Agency
(research program P1-0285, research project N1-0210).
Li is an affiliate member
of the Institute for Quantum Computing, University of Waterloo; his research was also supported
by the Simons Foundation Grant 851334.

We thank Dr.\ Sushil Singla for some helpful discussion.

\noindent
(Kuzma) Department of Mathematics, University of Primorska, Slovenia; Institute of Mathematics,
Physics, and Mechanics, Slovenia. E-mail: bojan.kuzma@famnit.upr.si

\medskip\noindent
(Li) Department of Mathematics, College of William \& Mary, Williamsburg, VA 23187, USA.
E-mail: ckli@math.wm.edu

\medskip\noindent
(Poon) Department of Mathematics, Embry-Riddle Aeronautical University, Prescott AZ 86301,
USA. E-mail: poon3de@erau.edu

\end{document}